\documentclass[letterpaper, 10pt, conference]{ieeeconf}
\IEEEoverridecommandlockouts 
\usepackage{amsmath,amssymb,url}
\usepackage{graphicx,subfigure}
\usepackage{color}
\usepackage{siunitx} 
\usepackage{algpseudocode,cases,booktabs}
\usepackage{siunitx}

\newcommand{\trs}[1]{\mathrm{tr}\ensuremath{[#1]}}

\renewcommand{\Re}{\ensuremath{\mathbb{R}}}

\newcommand{\D}{\ensuremath{\mathbf{D}}}

\newcommand{\kerG}{\mathbf{K}}
\newcommand{\mG}{\mathbf{m}}
\newcommand{\E}{\mathbb{E}}
\newcommand{\V}{\mathbb{V}}
\newcommand{\cov}{\operatorname{cov}}

\newcommand{\algrule}[1][.2pt]{\par\vskip.2\baselineskip\hrule height #1\par\vskip.2\baselineskip}

\title{\LARGE \bf
Adaptive Learning Kalman Filter with Gaussian Process}

\author{Taeyoung Lee\authorrefmark{1}
    \thanks{Taeyoung Lee, Mechanical and Aerospace Engineering, The George Washington University, Washington DC 20052 {\tt tylee@gwu.edu}}%
    \thanks{\textsuperscript{\footnotesize\ensuremath{*}}This research has been supported in part by NSF under the grant CNS-1837382, and AFOSR under the grant FA9550-18-1-0288.}
}

\graphicspath{{./Figs/}}

\newtheorem{definition}{Definition}

\newtheorem{prop}{Proposition}

\newtheorem{assumption}{Assumption}

\begin{document}
\allowdisplaybreaks
\maketitle \thispagestyle{empty} \pagestyle{empty}

\begin{abstract}
    This paper presents an adaptive Kalman filter for a linear dynamic system perturbed by an additive disturbance. 
    The objective is to estimate both of the state and the unknown disturbance concurrently, while learning the disturbance as a stochastic process of the state vector.
    This is achieved by estimating the state according to the extended Kalman filtering applied to the marginal distribution of the state, and by estimating the disturbance from a backward smoothing technique. 
    The corresponding pair of the estimated states and disturbances are fetched to a Gaussian process, which is constantly updated to resemble the disturbance process.
    The unique feature is that all of uncertainties in the estimated state and disturbance are accounted throughout the learning process. 
    The efficacy of the proposed approach is illustrated by a numerical example. 
\end{abstract}

\section{Introduction}

Kalman filters require that the system dynamics and its stochastic properties are exactly and completely given in prior. 
For example, an inaccurate noise covariance matrix results in sub-optimal performances or even divergence of error. 
To overcome these, various adaptive Kalman filters have been proposed~\cite{mehra1970identification,mehra1972approaches}.
For example, as the statistics of process noise are particularly challenging to obtain in prior, those are estimated online from the observed data~\cite{sarkka2009recursive,myers1976adaptive}.
Or, the optimal Kalman gain is directly estimated without estimating noise statistics~\cite{mehra1972approaches}. 
These have been applied to navigation systems~\cite{mohamed1999adaptive}, and visual object tracking~\cite{weng2006video}.
However, these approaches focus on parametric uncertainties, such as the covariance matrix of noise, and they do not handle unmodeled dynamics or disturbances that are dependent of the state.   

In machine learning, Gaussian processes have been widely used for stochastic modeling~\cite{williams2006gaussian}. 
It is defined as a stochastic process where any collection of those random variables is jointly Gaussian, 
and it is completely described by second-order statistics.
As such, it is often characterized by the covariance function, or the kernel function that describes the similarity between two input points. 
Gaussian processes can be also considered as a distribution over function on a continuous domain. 
In contrast to Bayesian learning with neural network~\cite{kendall2017uncertainties}, Gaussian processes inherit various properties of the normal distribution, and training or regression is completed explicitly without any iteration. 
Recently, it has been utilized for learning-based control of a nonlinear system~\cite{berkenkamp2016safe}, and reinforcement learning~\cite{engel2005reinforcement}.

This paper presents an adaptive Kalman filter that can deal with non-parametric, additive disturbances of a linear system, which is considered to be dependent of the state.
We aim to estimate the state and the disturbance concurrently, while modeling the disturbance function as a Gaussian process of the state.

First, the Gaussian process is extended to handle uncertainties in the input. 
The training data of any Gaussian process is composed of a set of input and output pairs. 
As the input state and the value of disturbance are estimated with uncertainties, the standard formulation of Gaussian processes with exact inputs cannot be directly applied here. 
We present an extended Gaussian process whose kernel function is adjusted to account the effects of noisy input data. 

Next, an adaptive learning Kalman filter is proposed by integrating forward filtering, backward smoothing, and learning. 
The forward filtering is to construct an estimate of the current state conditioned by all of available measurements, 
and the backward smoothing to update the estimate of the prior states using the current measurement.
This is followed by the learning process to augment and revise the training data set of an extended Gaussian process, which is updated to represent the disturbance more accurately. 

The unique property of the proposed adaptive learning Kalman filter is that the problem of state estimation is integrated with the learning process of the disturbance, while gauging the level of uncertainties between them. 
Such integration of learning and estimation has been unprecedented. 
Unless all of the elements of a state are measured directly, estimating the state is inherently coupled with learning the disturbance, as the prediction step of Bayesian estimation depends on the knowledge of disturbance, and also as the input to the disturbance function is only available through the current estimate of the state. 
We address this issue by thoroughly utilizing the extended Gaussian process, which is used to improve the state estimate while being refined from the improved estimate. 

Another desirable feature, especially for learning is that we can evaluate the confidence in the learned model depending on a selected input domain.  
This is particularly useful when utilizing the learned model beyond estimation, such as stochastic optimization or feedback controls.
A numerical example illustrates the state is successfully estimated in the presence of a state-dependent disturbance, which is currently estimated with an increasing accuracy and confidence.

\section{Extended Gaussian Process}

In this section, we first extend the Gaussian process such that it can deal with noisy and correlated data.
This is to incorporate uncertainties in the estimated state over the learning process of the proposed Kalman filter. 
Throughout this paper, we consider real, scalar valued Gaussian processes, and an extension for vector valued processes is available in~\cite{alvarez2012kernels}.
Also, $x\sim\mathcal{N}(\mu,\Sigma)$ denotes that a random variable $x$ is distributed according to the Gaussian distribution with the mean $\mu$ and the variance $\Sigma$ of appropriate dimensions. 
The corresponding density value is written as $\mathcal{N}(x|\mu,\Sigma)$.  

\subsection{Gaussian Process}

A Gaussian process is a stochastic process, defined such that any finite number of collection is jointly Gaussian~\cite{RasWilBK06}.
It is completely described by second-order statistics as follows. 
Define a mean function $\mG(x):\Re^n\rightarrow\Re$ and a positive-definite covariance function $\kerG(x,x'):\Re^n\times\Re^n\rightarrow\Re$, which is referred to as a kernel function.   
The corresponding Gaussian process is denoted by
\begin{align}
    g(x) \sim \mathcal{G}(\mG(x), \kerG(x,x')), \label{eqn:GP}
\end{align}

\subsection{Regression with Output Noise}

Let $\mathcal{D} = \{(x_i,g_i,\sigma_{g_i})\}_{i\in{1,\ldots N}} $ be a set of data, where $g_i\in\Re$ is a sample value of $g(x)$ when $x=x_i$, after corrupted by an additive, independent noise.
More explicitly, 
\begin{align}
    g_i \sim g(x_i) + \epsilon_{g_i},\label{eqn:g_pert}
\end{align}
with $\epsilon_{g_i} \sim \mathcal{N} (0,\sigma_{g_i}^2)$. 

Define $\mathbf{g}, \mathbf{x}$, and $\mG(\mathbf{x})\in\Re^N$ be the concatenation of $g_i$, $x_i$ and $\mG(x_i)$ for $i\in\{1,\ldots, N\}$, respectively. 
Also, let the matrix $\kerG(\mathbf{x},\mathbf{x})\in\Re^{N\times N}$ be defied such that its $i,j$-th element is $\kerG(x_i,x_j)$,
and let $\Sigma_{\mathbf{g}}= \mathrm{diag}[\sigma_{g_1}^2,\ldots, \sigma_{g_N}^2]\in\Re^{N \times N}$. 
The regression equation  for $g_*$ is 
\begin{align}
    g_* | \mathcal{D},x_* & \sim \mathcal{N}( \mG_* + \kerG_{*\mathbf{x}} (\kerG_{\mathbf{x}\mathbf{x}}+\Sigma_{\mathbf{g}})^{-1} (\mathbf{g}-\mG_\mathbf{x}),\nonumber\\
                          & \quad \kerG_{**} - \kerG_{* \mathbf{x}} (\kerG_{\mathbf{x}\mathbf{x}}+\Sigma_{\mathbf{g}})^{-1} \kerG_{\mathbf{x}*}), \label{eqn:GP_reg}
\end{align}
where the subscripts for $\mG$ and $\kerG$ denote the input arguments, e.g., $\kerG_{*\mathbf{x}} = \kerG(x_*, \mathbf{x})\in\Re^{1\times N}$. 

\subsection{Effects of Uncertain Inputs}

The preceding standard formulation of the Gaussian process assumes that the state vector for the data set $\mathbf{x}$ and the state for the regression $x_*$ are noise-free and uncorrelated. 
This is not desirable for the proposed adaptive learning Kalman filter, as the data set is an estimate of the possibly correlated state vector conditioned by measurements.

Several approaches have been considered to formulate a Gaussian process with uncertain inputs. 
In~\cite{dallaire2009learning}, an analytical expression for the expected value of a specific kernel is constructed for uncorrelated data. 
In~\cite{girard2003gaussian}, uncertainties in the state for the regression $x_*$ are incorporated by computing the first and the second moment of $g_*$. 
The uncertainties in the input is transformed to output noise in~\cite{mchutchon2011gaussian}.
Here we extend the approach of~\cite{girard2003gaussian} developed for uncertainties in $x_*$ to possibly correlated, uncertainty data set as follows. 
For simplicity, the output noise is not considered in this subsection, and it will be included later when formulating the extended Gaussian process formally.  

Suppose that for any $i\in\{1,\ldots, N\}$, the state $x_i$ in the data set follows a Gaussian distribution.
More specifically, $g_i$ is sampled from $g(x_i)$ where $x_i\sim N(\bar x_i, P_i)$ for a given mean $\bar x_i$ and a covariance $P_i\in\Re^{n\times n}$.
We have
\begin{align*}
    p(g_i) = \int_{\Re^n} p(g_i, x_i) dx_i.
\end{align*}
Since $p(g_i, x_i) = p(g_i|x_i) p(x_i)$, 
\begin{align}
    p(g_i) = \int_{\Re^n} \mathcal{N}(g_i|\mG(x_i),\kerG(x_i,x_i)) \mathcal{N}(x_i|\bar x_i, P_i) dx_i.\label{eqn:g_i}
\end{align}
Consequently, $g_i$ is not Gaussian in general.
Instead we show that the mean and the covariance of $g_i$ can be approximated as follows. 

\begin{prop}\label{prop:g_ij}
    Consider a set of random variables $\{g_1,\ldots, g_n\}$ distributed according to \eqref{eqn:g_i}, where $\{x_1,\ldots, x_N\}$ is jointly Gaussian with $\E[x_i]=\bar x_i\in\Re^n$, $\V[x_i]=P_i\in\Re^{3\times 3}$ and $\cov[x_i,x_j] = P_{ij}\in\Re^{n\times n}$.
    The mean and the covariance of $g_i$ are given by
    \begin{align}
        \E[g_i] & = \mG(\bar x_i) + \frac{1}{2}\trs{P_i\D^2 \mG(\bar x_i)}  + \mathcal{O}(\|x_i-\bar x_i\|^4), \label{eqn:Eg_i}\\
        \cov[g_i,g_j] & = \kerG(\bar x_i, \bar x_j) + \frac{1}{2}\trs{\D^2 \kerG(\bar x_i,\bar x_j) \mathbf{P}_{ij}} \nonumber\\
                                                    & \quad +  \trs{\D \mG(\bar x_i) \D\mG(\bar x_j)^T P_{ij}^T} \nonumber\\
                                                    & \quad -\frac{1}{4} \trs{\D^2 \mG(\bar x_i)P_i} \trs{\D^2 \mG(\bar x_j)P_j} \nonumber \\
                                                    & \quad + \mathcal{O}(\|x_i-\bar x_i\|^4), \label{eqn:covg_ij}
    \end{align}
where $\D$ denotes the derivatives, e.g.,
    \begin{align*}
        \D^2 \mG(\bar x_i) = \frac{\partial^2 \mG(x)}{\partial x \partial x}\bigg|_{x=\bar x_i},
    \end{align*}
    and $\mathbf{P}_{ij}\in\Re^{2n\times 2n}$ is defined as
    \begin{align}
        \mathbf{P}_{ij} = \begin{bmatrix} P_i & P_{ij} \\ P_{ji} & P_j \end{bmatrix}. \label{eqn:PP_ij}
    \end{align}
\end{prop}
\begin{proof}
According to the law of total expectation, namely $\E[Y] = \E_X[\E_Y[Y|X]]$~\cite{AksHenBJ04}, we have
\begin{align*}
    \E[g_i] =  \E_{x_i} [ \E[g(x_i)|x_i=\chi]] = \int_{\Re^n} \mG(\chi) \mathcal{N}(\chi|\bar x_i, P_i) d
    \chi.
\end{align*}
The Tayler series expansion about $\chi=\bar x_i$ yields \eqref{eqn:Eg_i}, and it becomes of the fourth-order, as the third order moment of any Gaussian distribution is zero. 

Similarly, from the law of the total covariance, namely $\cov[X,Y] = \E[\cov[X,Y|Z]] + \cov[\E[X|Z],\E[Y|Z]]$~\cite{AksHenBJ04},
\begin{align}
    \cov[g_i,g_j] & = \E[\cov[g_i,g_j| x_i = \chi_i, x_j=\chi_j]] \nonumber \\
                  & \quad + \cov[\E[g_i|x_i=\chi_i],\E[g_j|x_j=\chi_j]]. \label{eqn:covg_ij_0}
\end{align}
The first term of the right hand size of \eqref{eqn:covg_ij_0} is
\begin{align*}
    & \E [\cov[g_i,g_j| x_i = \chi_i, x_j=\chi_j]] \\
       & = \iint \kerG(\chi_i,\chi_j) \mathcal{N}((\chi_i,\chi_j)|(\bar x_i,\bar x_j), \mathbf{P}_{ij} ) d\chi_i d\chi_j\\
       & = \kerG(\bar x_i,\bar x_j)  + \frac{1}{2}\trs{\D^2_1 \kerG(\bar x_i,\bar x_j) P_i} \\
       & \quad + \trs{\D_1\D_2 \kerG(\bar x_i,\bar x_i) P_{ij}^T} + \frac{1}{2}\trs{\D^2_2\kerG(\bar x_i,\bar x_j)P_j} + \mathcal{O}(4),
\end{align*}
where $\D_1$ denotes the derivatives with respect to the first input argument, and $\D_2$ is defined similarly. 
For instance, the $k,l$-th element of $\D_1\D_2 \kerG(\bar x_i,\bar x_j)\in\Re^{n\times n}$ is given by
\begin{align*}
    [ \D_1\D_2 \kerG(\bar x_i,\bar x_j)]_{k,l} = \frac{\partial^2 \kerG(\chi_i,\chi_j)}{\partial \chi_{ik}\partial \chi_{jl}}\bigg|_{\chi_{i}=\bar x_i, \chi_j=\bar x_j},
\end{align*}
where $\chi_{ik}$ and $\chi_{jl}\in\Re$ denotes the $k$-th element of $\chi_i$, and the $l$-th element of $\chi_j$, respectively. 
The above reduces to the first two terms of the right hand side of \eqref{eqn:covg_ij} with \eqref{eqn:PP_ij}.

Next, the second term of \eqref{eqn:covg_ij_0} is 
\begin{align*}
    & \cov[\E[g_i|x_i=\chi_i],\E[g_j|x_j=\chi_j]]\\
    & = \E[(\mG(\chi_i)-\E[\mG(\chi_i)]) (\mG(\chi_j)-\E[\mG(\chi_j)])] \\
    & = \trs{\D \mG(\bar x_i) \D\mG(\bar x_j)^T P_{ij}^T}  \\
    & \quad - \frac{1}{4} \trs{\D^2 \mG(\bar x_i)P_i} \trs{\D^2 \mG(\bar x_j)P_j} + \mathcal{O}(4),
\end{align*}
which corresponds to the remaining part of \eqref{eqn:covg_ij}.
\end{proof}

\subsection{Extended Gaussian Process}

The above proposition states that the mean and the covariance of $g$ is approximated by \eqref{eqn:Eg_i} and \eqref{eqn:covg_ij} up to the fourth order of the perturbation of the input state. 
Let $\tilde\mG:\Re^n\rightarrow\Re$ and $\tilde\kerG:\Re^n\times\Re^n\rightarrow\Re$ be the corresponding approximation:
\begin{align}
    \tilde\mG(x_i) & = \mG(\bar x_i) + \frac{1}{2}\trs{P_i\D^2 \mG(\bar x_i)}, \label{eqn:tilde_m}\\
    \tilde\kerG(x_i,x_j) & = \kerG(\bar x_i, \bar x_j) + \frac{1}{2}\trs{\D^2 \kerG(\bar x_i,\bar x_j) \mathbf{P}_{ij}} \nonumber\\
                  & \quad +  \trs{\D \mG(\bar x_i) \D\mG(\bar x_j)^T P_{ij}^T} \nonumber\\
                  & \quad -\frac{1}{4} \trs{\D^2 \mG(\bar x_i)P_i} \trs{\D^2 \mG(\bar x_j)P_j} \label{eqn:tilde_K}.
\end{align}
The Gaussian process with the above perturbed mean and kernel is defined as the \textit{extended} Gaussian process.
\begin{definition}
Consider the Gaussian process given at \eqref{eqn:GP}. 
Assume that any collection of the input is jointly Gaussian with a prescribed mean and variance.
The corresponding extended Gaussian distribution is defined as
\begin{align}
    g(x) \sim \tilde{\mathcal{G}}(\tilde\mG(x), \tilde{\kerG}(x,x')). \label{eqn:EGP}
\end{align}
\end{definition}

In short, the extended Gaussian approximates the standard Gaussian process perturbed by noisy input, namely \eqref{eqn:g_i}, up to the second moments. 

\subsection{Regression of Extended Gaussian Process}

The desirable feature is that all of properties of the standard Gaussian process hold with the perturbed mean and kernel. 
For instance, suppose the output is perturbed as in \eqref{eqn:g_pert}. 
The training data set of the extended Gaussian process is given by $\tilde{\mathcal{D}} = \{ \bar x_i, g_i, P_{ij}, \sigma_{g_i} \}_{i,j\in\{1,N\}}$.
We have
\begin{align}
    \mathbf{g} \sim \mathcal{N} (\tilde\mG(\mathbf{g}), \tilde\kerG(\mathbf{x},\mathbf{x}) + \Sigma_{\mathbf{g}}).
\end{align}

For regression, let $g_*\in\Re$ be a sample value for $x=x_*$, where $x_*$ is jointly Gaussian with $\mathbf{x}$. 
Specifically, $x_*\sim\mathcal{N}(\bar x_*, P_*)$ with $\cov(x_i,x_*)=P_{i*}\in\Re^{n\times n}$ for $i\in\{1,\ldots, n\}$.
The joint distribution for $(\mathbf{g},g_*)$ is
\begin{align}
    \begin{bmatrix} 
        \mathbf{g} \\ g_*
    \end{bmatrix}
    & = 
    \mathcal{N}
    \left( \begin{bmatrix} 
            \tilde\mG(\mathbf{x}) \\
            \tilde\mG(x_*)
        \end{bmatrix}
    ,
    \begin{bmatrix}
        \tilde\kerG(\mathbf{x}, \mathbf{x}) + \Sigma_{\mathbf{g}} & \tilde\kerG(\mathbf{x}, x_*)\\
        \tilde\kerG(x_*, \mathbf{x}) & \tilde\kerG(x_*, x_*)\\
    \end{bmatrix}
\right).\label{eqn:ggg*_EGP}
\end{align}
Let the input data be $\mathcal{I}_*=(\bar x_*, P_*, P_{1*},\ldots, P_{N*})$.
Similar with \eqref{eqn:GP_reg},
\begin{align}
    g_* | \tilde{\mathcal{D}}, \mathcal{I}_*  \sim \mathcal{N}( \tilde \mu(x_*) , \tilde\Sigma(x_*)),
\end{align}
where the mean and the covariance of the output are defined as
\begin{align}
    \tilde \mu(x_*) & = \tilde \mG_* + \tilde \kerG_{*\mathbf{x}} (\tilde \kerG_{\mathbf{x}\mathbf{x}}+\Sigma_{\mathbf{g}})^{-1} (\mathbf{g}-\tilde \mG_{\mathbf{x}}), \label{eqn:tilde_mu}\\
    \tilde\Sigma(x_*) & = \tilde \kerG_{**} - \tilde \kerG_{* \mathbf{x}} (\tilde \kerG_{\mathbf{x}\mathbf{x}}+\Sigma_{\mathbf{g}})^{-1} \tilde \kerG_{\mathbf{x}*}) .\label{eqn:tilde_sigma}
\end{align}

The above expressions require that the matrix composed of the kernel function $\tilde{\kerG}_{\mathbf{xx}}$ be positive-definite. 
As presented at Proposition~\ref{prop:g_ij}, it is a fourth-order approximation of the covariance matrix of $(g_1,\ldots, g_N)$. 
Therefore, there is no guarantee that the kernel $\tilde \kerG$ is positive-definite, especially if $P_{i}$ is large. 

Various techniques have been considered to deal with indefinite kernel functions.
We adopt the technique referred to as spectrum flip~\cite{chen2009similarity}. 
For consistency in regression, this method is applied to the covariance matrix for the concatenated training data and the regression input.
Let the covariance matrix of~\eqref{eqn:ggg*_EGP} be $\mathbf{\Sigma}\in\Re^{N+1\times N+1}$, which is symmetric, but not necessarily positive-definite.  
Suppose the eigendecomposition of $\mathbf{\Sigma}$ be $\mathbf{\Sigma} = V \Lambda V^T$, where $V\in\Re^{N+1\times N+1}$ is composed of normalized orthonormal eigenvectors, and $\Lambda\in\Re^{N+1\times N+1}$ is the diagonal matrix whose diagonal elements are the corresponding eigenvalues. 
The spectrally flipped covariance is given by $\mathbf{\Sigma}'= V\sqrt{\Lambda^2}V^T$, which replaces the covariance of \eqref{eqn:ggg*_EGP} for regression, e.g., the flipped $\tilde\kerG_{**}$ corresponds to the $N+1,N+1$-th element of $\mathbf{\Sigma}'$.
This can be interpreted as formulating a kernel on the pseudo-Euclidean space.

\subsection{Numerical Example}

We consider a numerical example for $g(x) = \sin 4\pi x$.
The training data are chose as
\begin{gather*}
    \bar x_i =0.1\times i, \quad P_i = 0.01^2, \quad P_{ij}=0,
    \quad \sigma_{g_i} =0.01,
\end{gather*}
for $0\leq i\neq j \leq 10$, resulting in $N=11$ data points. 
The value of $x_i$ and $g_i$ is sampled from the corresponding  Gaussian distribution. 
For regression, $x_*$ is varied from $0$ to $1$ with $P_*=0.01^2$ and $P_{i*}=0$.
For the kernel function, we use the squared exponential function given at Appendix A, with the hyperparameters $L=0.1I_{N\times N}$, $\sigma_f=1$,  $\sigma_n=0.1$.

Figure \ref{fig:EGP} illustrates the results of regression, where the true function value is denoted by a red line, and the output of the extended Gaussian process is denoted by a blue curve with $3\sigma$ bounds denoted by gray shades.  
The training data are marked with blue stars. 

In particular, Figure \ref{fig:EGP_a} is when the variance of the fifth data point is increased to $P_4=0.1^2$, i.e., $x_4\sim\mathcal{N}(0.4, 0.1^2)$. 
The corresponding sample value for $(x_4,g_4)=(0.48, -1.96)$ is marked by a blue circle around a star. 
Due to the large uncertainties at $x_4$, the output of the extended Gaussian process also exhibits increased uncertainties around $\bar x_4 =0.4$. 

Similarly, Figure \ref{fig:EGP_b} shows the results when the variance of the regression point $x_*$ is increased to $P_*=0.04^2$ for $0.7\leq x_* \leq 0.75$. 
The variance of the output is increased accordingly over the same range. 
These illustrate the capability of the extended Gaussian process in handling uncertainties in the input for both of training data and regression.

\begin{figure}
    \centerline{
        \subfigure[Regression with increased uncertaintes in the training data]{
            \includegraphics[width=0.7\columnwidth]{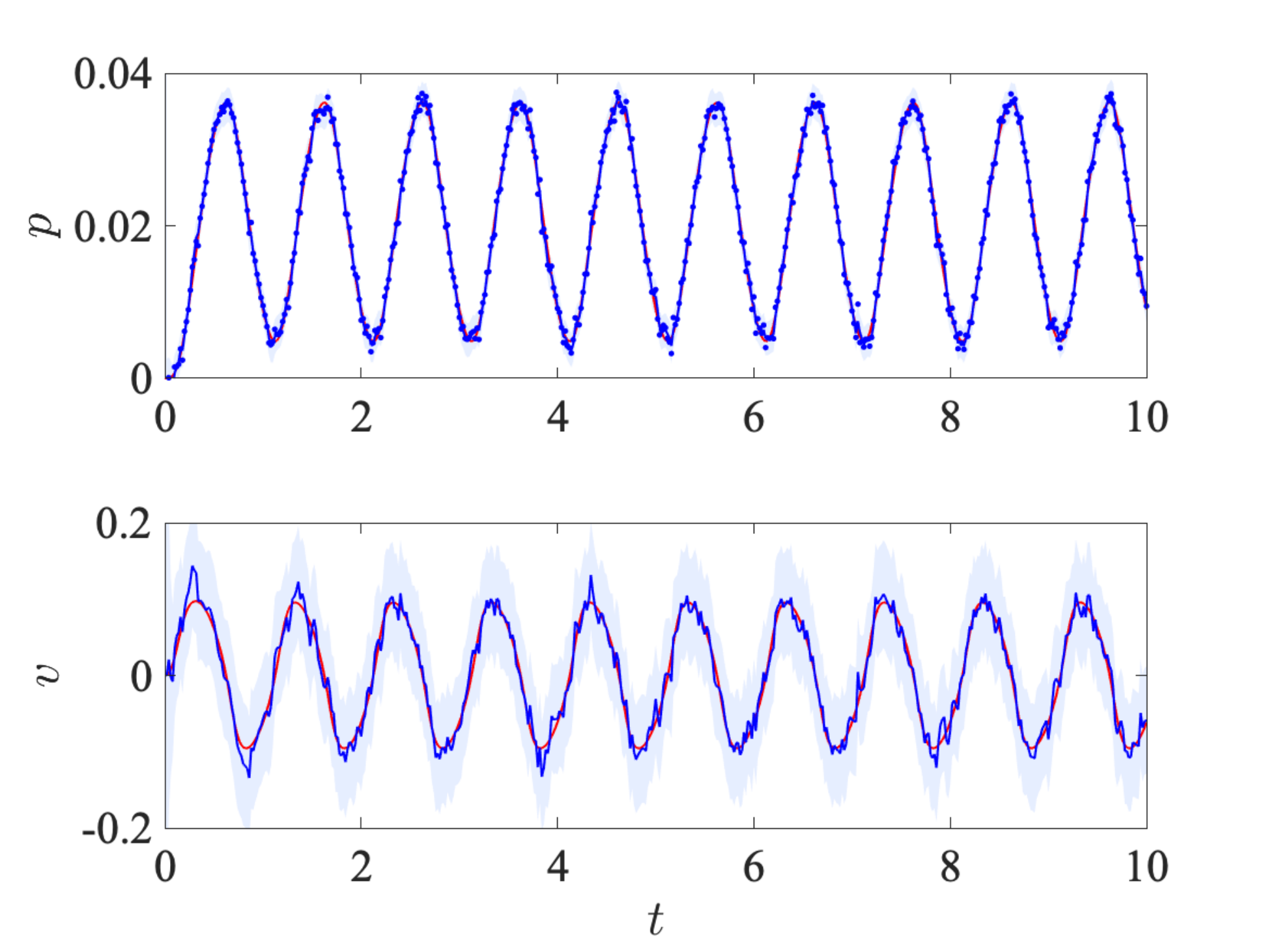}
        \label{fig:EGP_a}}
    }
    \centerline{
        \subfigure[Regression with increased uncertainties in th e input]{
            \includegraphics[width=0.7\columnwidth]{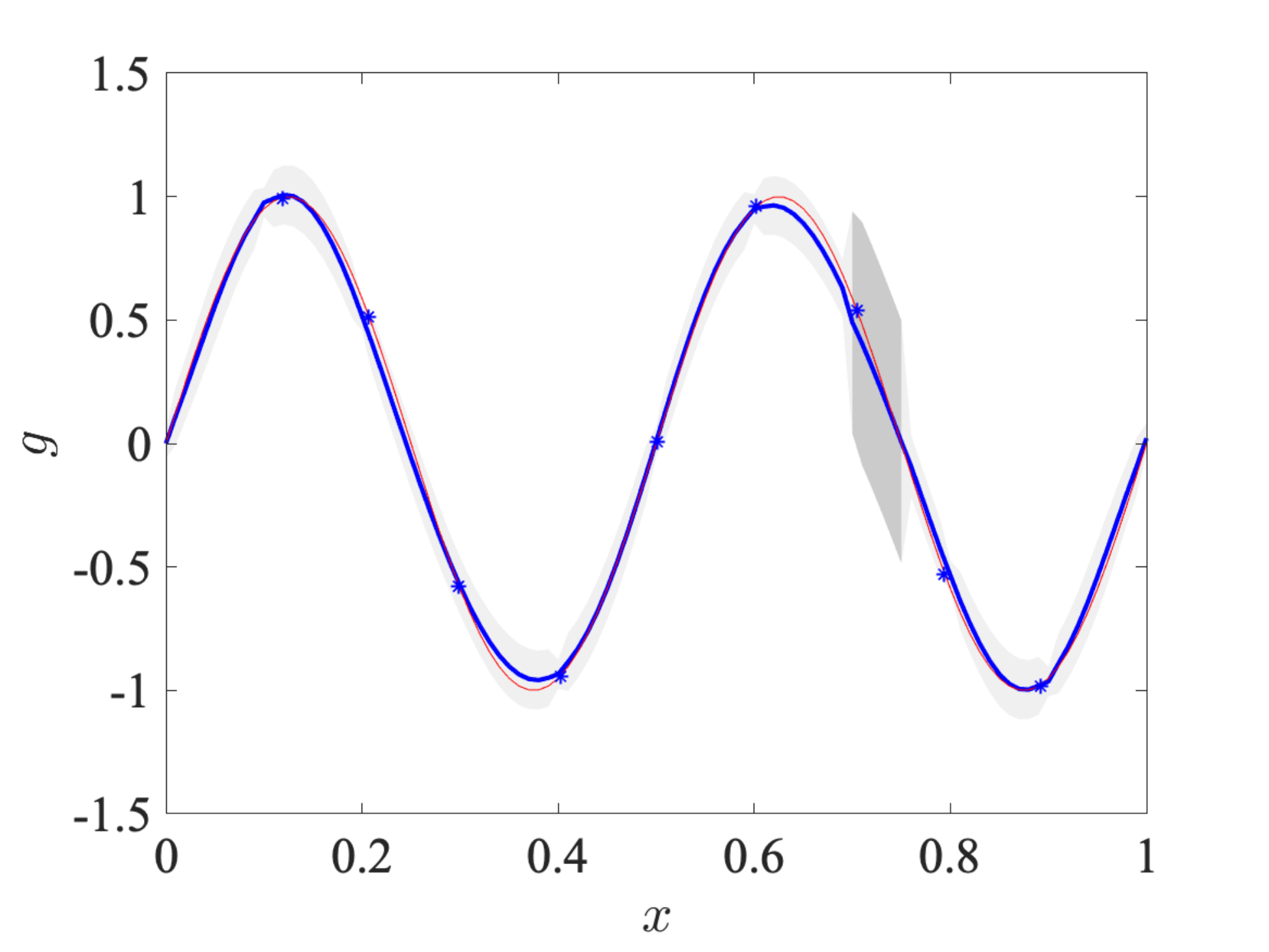}
        \label{fig:EGP_b}}
    }
    \caption{Numerical example for the proposed extended Gaussian process}\label{fig:EGP}
\end{figure}

\section{Adaptive Learning Kalman Filter}

In this section, we propose an adaptive learning Kalman filter for a linear time-varying system perturbed by an additive disturbance.
The key idea is that the current knowledge of the disturbance is represented by the extended Gaussian process presented in the prior section,
and it is refined whenever new measurements become available. 


\subsection{Problem Formulation}

Consider a  discrete, time-varying system given by
\begin{align}
    x_{k+1} & = A_k x_k + B_k u_k + G_k g(x_k) + w_k,\label{eqn:x_kp} \\
    z_k & = H_k x_k + v_k,\label{eqn:z_k}
\end{align}
where $x_k\in\Re^n$, $u_k\in\Re^m$, and $g(x_k)\in\Re^p$ are the state, the control input, and the state-dependent disturbance, respectively.
The sensor measurement is given by $z_k\in\Re^p$.
The process noise and the measurement noise are denoted by $w_k\in\Re^n$ and $v_k\in\Re^n$, respectively, with $w_k\sim\mathcal{N}(0, Q_k)$ and $v_k\sim\mathcal{N}(0,R_k)$ for symmetric, positive-definite matrices $Q_k\in\Re^{n\times n}$ and $R_k\in\Re^{q\times q}$.
The matrices $A_k,B_k,G_k$ and $H_k$ are of appropriate dimensions.

The initial state follows $x_0\sim\mathcal{N}(\bar x_0, P_0)$ for the given mean $\bar x_0\in\Re^n$ and the covariance $P_0\in\Re^{n\times n}$.
The initial state and the noise vectors at every step are mutually independent.  

We assume $g(x)$ follows a Gaussian process as in~\eqref{eqn:GP}.
Consequently, when the training data are uncertain, it can be modeled as an extended Gaussian process \eqref{eqn:EGP}.
Instead of distinguishing the true system from the learned model, it is considered that the initial estimate of $g$ is conservative enough to enclose the true disturbance as one of its sample process.
In other words, the variance without any training data, namely $\kerG(x,x)$ is sufficiently large. 
We further make the following assumption. 
\begin{assumption}\label{asm:G}
    The matrix $G_k\in\Re^{n\times p}$ has the full column rank for all $k$.
\end{assumption}
This is to ensure that we can infer the value of $g(x_k)$ from the estimates of $x_k$ and $x_{k+1}$ in the learning process. 

The proposed adaptive learning Kalman filter is composed of three steps: prediction, correction, and learning.

\subsection{Prediction and Correction}

We first describe the stochastic property of the extended Gaussian process at the $k$-th step. 
Let the training data set at the $k$-th step be
\begin{align}
    \tilde{\mathcal{D}}_k = \{\bar x^S_j, \bar g_j, P^S_{ji}, \sigma_{g_j}\}_{i,j\in\{0,\ldots, k-1\}}, \label{eqn:Dk}
\end{align}
which is composed of the estimated value of the unknown disturbance $\bar g_j$ at a given state $\bar x^S_j$ with uncertainties represented by $\sigma_{g_j}$ and $P^S_{ji}$, respectively. 
The input data for regression at the $k$-th step is 
\begin{align}
    \mathcal{I}_k =\{ \bar x_k, P_k, \{P^S_{jk}\}_{k\in\{0,\ldots, k-1\}}\}.\label{eqn:Ik}
\end{align}

The training data set and the input date will be defined later at the learning step by using all of the measurements available at $t_k$,  namely $Z_k=[z_1,\ldots z_k]\in(\Re^q)^k$.
For the initial time, there is no training data available, and therefore $\tilde{\mathcal{D}}_0=\emptyset$, and the input data reduces to $\mathcal{I}_0=\{ \bar x_0, P_0\}$. 

From \eqref{eqn:tilde_mu} and \eqref{eqn:tilde_sigma},
\begin{align}
    \E[g_k] & = \tilde\mu(x_k),\label{eqn:Eg_k} \\
    \V[g_k] & = \tilde\Sigma(x_k),
\end{align}
which are computed from $\tilde{\mathcal{D}}_k$ and $\tilde{\mathcal{I}}_k$. 
By adopting the approaches of extended Kalman filters, we take the linearization of the mean function to obtain
\begin{align}
    \cov[g_k, x_k] = \D\tilde\mu(x_k) P_{k}, \label{eqn:covgx}
\end{align}
where $\D\tilde\mu(x_k)\in\Re^{q\times n}$ is the derivative of $\tilde\mu(x_k)$ with respect to the mean value of $x_k$. 

Let $x_k|Z_k \sim \mathcal{N}(\bar x_k, P_k)$. 
From \eqref{eqn:Eg_k}-\eqref{eqn:covgx}, it is straightforward to show that the joint distribution $x_{k+1},z_{k+1}|Z_k$ is given by
\begin{align}
    \begin{bmatrix}
        x_{k+1} | Z_k\\
        z_{k+1} | Z_k 
    \end{bmatrix}
    \sim \mathcal{N}(
    \begin{bmatrix}
        \bar x_{k+1}^-\\
        H_k \bar x_{k+1}^ - 
    \end{bmatrix} ,
    \begin{bmatrix}
        P_{k+1}^- & P_{k+1}^- H_{k+1}^T \\
        H_{k+1}P_{k+1}^ - & S_{k+1}
    \end{bmatrix}
    ),\label{eqn:xzZ}
\end{align}
where  $\bar x_{k+1}^-\Re^n$, $P_{k+1}^-\in\Re^{n\times n}$, and $S_{k+1}\in\Re^{q\times q}$ are 
\begin{align}
    \bar x^-_{k+1} & = A_k \bar x_k + B_k u_k + G_k\tilde\mu(x_k),\\
    P^-_{k+1} & = A_k P_k A_k^T + A_kP_kG_k^T (\D\tilde\mu(x_k))^T \nonumber\\
              & \quad + G_k\D\tilde\mu(x_k) P_k A_k^T + G_k\tilde\Sigma(x_k) G_k^T + Q_k,\\
    S_{k+1} & = H_{k+1} P_{k+1}^- H_{k+1}^T + R_{k+1}.
\end{align}

From \eqref{eqn:Gau_cond}, the posterior distribution conditioned by the measurement $z_{k+1}$ is given by
\begin{align}
    x_{k+1}|Z_{k+1} \sim \mathcal{N}(\bar x_{k+1}, P_{k+1}),\label{eqn:xkpZkp}
\end{align}
where the posterior mean $\bar x_{k+1}\in\Re^n$, covariance $P_{k+1}\in\Re^{n\times n}$, and the Kalman gain $K_{k+1}\in\Re^{n\times q}$ are
\begin{align}
    \bar x_{k+1} & = \bar x_{k+1}^- + K_{k+1} (z_{k+1} - H_{k+1}\bar x_{k+1}^-),\\
    P_{k+1} & =  (I_{n\times n}- K_{k+1} H_{k+1}) P_{k+1}^-,\\
    K_{k+1} & = P_{k+1}^- H_{k+1}^T S_{k+1}^{-1}.
\end{align}
These are essentially an extended Kalman filter for the marginal distribution of the state, developed with the properties of the extended Gaussian process given by \eqref{eqn:Eg_k}--\eqref{eqn:covgx}.
This is followed by the learning step described below.

\subsection{Learning}

As the uncertain term is represented by an extended Gaussian process, the learning step constitutes of refining and augmenting the training data with all of the measurement available. 
When a new measurement $z_{k+1}$ becomes available, the above correction step revises $x_{k+1}|Z_{k}$ to construct a new estimate $x_{k+1}|Z_{k+1}$, but the estimate for any of prior states is not updated. 
While this is reasonable for online state estimation with the Markov property, it is not desirable for the learning problem considered here, as the training data set is composed of the \textit{history} of estimated states.  

For the learning step, we first update the estimate for the prior states to construct $\{x_{j}|Z_{k+1}\}_{j\in\{0,\ldots, k+1\}}$.
Such problem of estimating past states conditioned by the current measurement  is referred to as smoothing~\cite{BryHo75}.

The smoothing problem is formulated as a backward recursive iteration, initiated with $x_{k+1}|Z_{k+1}\sim\mathcal{N}( \bar x^S_{k+1},  P^S_{k+1})$, where the superscript $S$ denotes the mean and the variance conditioned by all of the available measurements $Z_{k+1}$, estimated through the smoothing.  
From \eqref{eqn:xkpZkp},
\begin{align}
    \bar x^S_{k+1}=\bar x_{k+1},\quad P^S_{k+1}=P_{k+1}.\label{eqn:xSkp}
\end{align}

Next, we derive backward recursion equations.
For any $0\leq j \leq k$, suppose 
\begin{align}
    x_{j+1}|Z_{k+1}\sim\mathcal{N}(\bar x_{j+1}^S, P_{j+1}^S), \label{eqn:xZkp}
\end{align}
with the given mean and covariance $(\bar x_{j+1}^S, P_{j+1}^S)$.
From the definition of the conditional density, the joint distribution with the  state in the previous step is written as
\begin{align*}
    p(x_j,x_{j+1}|Z_{k+1}) & = p(x_j|x_{j+1},Z_{k+1}) p(x_{j+1}|Z_{k+1})\nonumber\\
                           & = p(x_j|x_{j+1},Z_{j}) p(x_{j+1}|Z_{k+1}),
\end{align*}
where we have used the Markov property that $x_j \perp (z_{j+1},\ldots, z_{k+1}) |x_{j+1}$ for the second equality. 
From \eqref{eqn:xZkp}, the last term is replaced by the outcome of the prior iteration as
\begin{align}
    p(x_j,x_{j+1}|Z_{k+1}) & = p(x_j|x_{j+1},Z_{j})\mathcal{N}(x_{j+1}| \bar x_{j+1}^S, P_{j+1}^S) ,\label{eqn:xxZ}
\end{align}

Next, we find the conditional distribution $p(x_j|x_{j+1},Z_{j})$ of the above expression, 
using its joint distribution given by
\begin{align*}
    p(x_j,x_{j+1}|Z_{j}) = p(x_{j+1}|x_j,Z_j) p(x_j|Z_j),
\end{align*}
which is not Gaussian in general. 
However, we have $x_j|Z_j\sim\mathcal{N}(\bar x_j, P_j)$ from the correction step. 
Similar with \eqref{eqn:xzZ}, it can be approximated by
\begin{align}
    \begin{bmatrix}
        x_{j} | Z_j\\
        x_{j+1} | Z_j 
    \end{bmatrix}
    \sim \mathcal{N}(
    \begin{bmatrix}
        \bar x_{j}\\
        \bar x'_{j+1}
    \end{bmatrix} ,
    \begin{bmatrix}
        P_{j} & P_j (A'_j)^T \\
        A'_j P_{j} & P'_{j+1}
    \end{bmatrix}
    ),
\end{align}
where  $\bar x'_{j+1}\in \Re^n$, and $A'_j, P'_{j+1}\in\Re^{n\times n}$ are
\begin{align*}
    \bar x'_{j+1} & = A_j \bar x_j + B_j u_j + G_j\tilde\mu(x_j),\\
    A'_j &  = A_j + G_j \D\tilde\mu(x_j),\\
    P'_{j+1} & = A_j P_j A_j^T + A_jP_jG_j^T (\D\tilde\mu(x_j))^T \nonumber\\
              & \quad + G_j\D\tilde\mu(x_j) P_j A_j^T + G_j\tilde\Sigma(x_j) G_j^T + Q_j.
\end{align*}
From \eqref{eqn:Gau_cond}, the conditional distribution is
\begin{align}
    x_j | x_{j+1},Z_j & \sim \mathcal{N}(
    \bar x_j + K'_j(x_{j+1} - \bar x'_{j+1}),\nonumber \\
                      & \quad (I_{n\times n} - K'_j A'_j ) P_j' ), \label{eqn:xxZ_normal}
\end{align}
where  $K'_j\in\Re^{n\times n}$ is 
\begin{align}
K'_j = P_j(A'_j)^T (P'_{j+1})^{-1}.
\end{align}

Finally, we substitute \eqref{eqn:xxZ_normal} to \eqref{eqn:xxZ}, and use the property of the Gaussian distribution, namely \eqref{eqn:Gau_cond},  to obtain
\begin{align}
    \bar x^S_{j} & = \E[x_j |Z_{k+1}] = \bar x_j + K'_j (\bar x^S_{j+1} - \bar x'_{j+1}), \label{eqn:xSj} \\
    P^S_j & = \V[x_j|Z_{k+1}] = (I_{n\times n} - K'_j A'_j ) P_j'  + K'_j P^S_{j+1}  K'_j,\\
    P^S_{j,j+1} & = \cov[x_j,x_{j+1}|Z_{k+1}] = K'_j P^S_{j+1}. \label{eqn:PSjj}
\end{align}
In short, these yield a backward recursion from $(\bar x^S_{j+1},P^S_{j+1})$ at \eqref{eqn:xZkp} to $(\bar x^S_{j},P^S_{j})$ at \eqref{eqn:xSj}--\eqref{eqn:PSjj}. 
Initiated by \eqref{eqn:xSkp}, we obtain the history of estimation $\{\bar x^S_j,  P^S_j, P^S_{j,j+1}\}_{j\in\{0,\ldots, k+1\}}$ conditioned by $Z_{k+1}$. 

These provide an estimate for the sample value of the Gaussian process. 
From \eqref{eqn:x_kp}, and Assumption \ref{asm:G}, the sample value of $g(x_j)$ satisfies
\begin{align*}
    g_j = G_j^{\dagger} (x_{j+1} - A_j x_j - B_j u_j + w_j),
\end{align*}
where $G_j^{\dagger}\in\Re^{p\times n}$ is the matrix pseudo-inverse given by $G_j^{\dagger} =(G_j^TG_j)^T G_j$. 
As a linear combination of jointly Gaussian variables follows another Gaussian distribution, $g_j|Z_{k+1}$ is Gaussian with
\begin{align}
    \bar g_j & = \E[g_j | Z_{k+1}] = G_j^{\dagger} (\bar x^S_{j+1} - A_j \bar x^S_j - B_j u_j ),\label{eqn:bargj}\\
    \sigma_{g_j} & = \V[g_j |Z_{k+1}] = G_j^{\dagger} ( P^S_{j+1} - A_j P^S_{j,j+1} - P^S_{j+1,j}A_j^T \nonumber\\
                 & \quad + A_j P^S_j A_j^T + Q_j ) (G_j^{\dagger})^T,\label{eqn:sigmagj}
\end{align}
for $j\in\{0,\ldots, k\}$. 
From \eqref{eqn:xSj}--\eqref{eqn:sigmagj}, we can construct the updated training data set $\tilde{\mathcal{D}}_k$ and the input data set, defined by \eqref{eqn:Dk} and \eqref{eqn:Ik}, respectively, for the next prediction and correction steps.

The overall procedure of the proposed adaptive learning Kalman filter is summarized at Table \ref{tab:ALKF}.
    
\begin{table}
\caption{Adaptive Learning Kalman Filter}\label{tab:ALKF}
\begin{algorithmic}[1]
\algrule[0.8pt]
\Procedure{Adaptive Learning Kalman Filter}{}
\algrule
\State $k=0$, $x_0\sim\mathcal{N}(\bar x_0, P_0)$, $\tilde{\mathcal{D}}_k=\emptyset$
	\Repeat
    \State $[\bar x_{k+1}, P_{k+1}]=${\fontshape{sc}\selectfont Kalman Filter}$(\bar x_k, P_k, \tilde{\mathcal{D}}_k, z_{k+1})$
    \State $\tilde{\mathcal{D}}_{k+1}=${\fontshape{sc}\selectfont Learning}$(\bar x_{k+1}, P_{k+1}, z_{k+1})$
		\State $k=k+1$
	\Until{terminal time is reached}
\EndProcedure
\algrule
\Procedure{$[\bar x_{k+1}, P_{k+1}]=${\fontshape{sc}\selectfont Kalman Filter}}{$\bar x_k, P_k, \tilde{\mathcal{D}}_k, z_{k+1}$}
    \State Gaussian process regression with \eqref{eqn:Eg_k}--\eqref{eqn:covgx}
    \State Prediction with \eqref{eqn:xzZ}
    \State Correction with \eqref{eqn:xkpZkp}
\EndProcedure
\algrule
\Procedure{$\tilde{\mathcal{D}}_{k+1}=${\fontshape{sc}\selectfont Learning}}{$\bar x_{k+1}, P_{k+1}, z_{k+1}$}
\State $(\bar x^S_{k+1}, P^S_{k+1}) = (\bar x_{k+1}, P_{k+1})$
\For{$j \gets k,\ldots 1$}
\State Compute $\bar x^S_j, P^S_j, P^S_{j,j+1}$ from \eqref{eqn:xSj}--\eqref{eqn:PSjj}
\State Compute $\bar g_j, \sigma_{g_j}$ from \eqref{eqn:bargj}--\eqref{eqn:sigmagj}
\EndFor
\State Set $\tilde{\mathcal{D}}_{k+1}=\{\bar x^S_j, \bar g_j, P^S_{ji}, \sigma_{g_j}\}_{i,j\in\{0,\ldots, k\}}$
\EndProcedure
\algrule[0.8pt]
\end{algorithmic}
\end{table}

\section{Numerical Example}

Consider a one-dimensional vehicle model moving along a straight line. 
The equations of motion are given by
\begin{align*}
    \dot p & = v,\\
    \dot v & =  u + \Delta(v),
\end{align*}
where $p,v\in\Re$ denote the position and the velocity of the vehicle, respectively. 
There is a control force and an unknown disturbance force defined as 
\[
    u(t) = \sin 2\pi t,\quad
    \Delta(v) = -100 |v| v.
\]
The disturbance corresponds to a drag acting opposite to the motion of the vehicle with the magnitude proportional to $v^2$. 
It is assumed that the position is measured by a sensor. 

Let the state vector be $x=[p, v]\in\Re^2$, and let the fixed step size be $h>0$. 
The above equations of motion are discretized as \eqref{eqn:x_kp} and \eqref{eqn:z_k} with
\begin{align*}
    A_k = \begin{bmatrix} 1 & h \\ 0 & 1 \end{bmatrix},\quad
    B_k = G_k = \begin{bmatrix} \frac{h^2}{2} \\ h \end{bmatrix},\quad
    H_k = \begin{bmatrix} 1 & 0 \end{bmatrix}.
\end{align*}
The time step is $h=0.02$.
The noise covariance matrices are chosen as $Q_k = \mathrm{diag}[0, 0.01^2]$,  $R_k=0.001^2$. 
The initial estimate is $\bar x_0=[0,0]$ and $P_0=0.2^2I_{2\times 2}$.

For the extended Gaussian process, the kernel function is chosen as the squared exponential function presented at Appendix B with the hyperparameters $l=0.04$, $\sigma_f=1$, and $\sigma_n=0.1$.
The mean function is chosen as zero-valued everywhere.

\begin{figure}
    \centerline{
        \subfigure[Kalman filter]{
            \includegraphics[width=0.9\columnwidth]{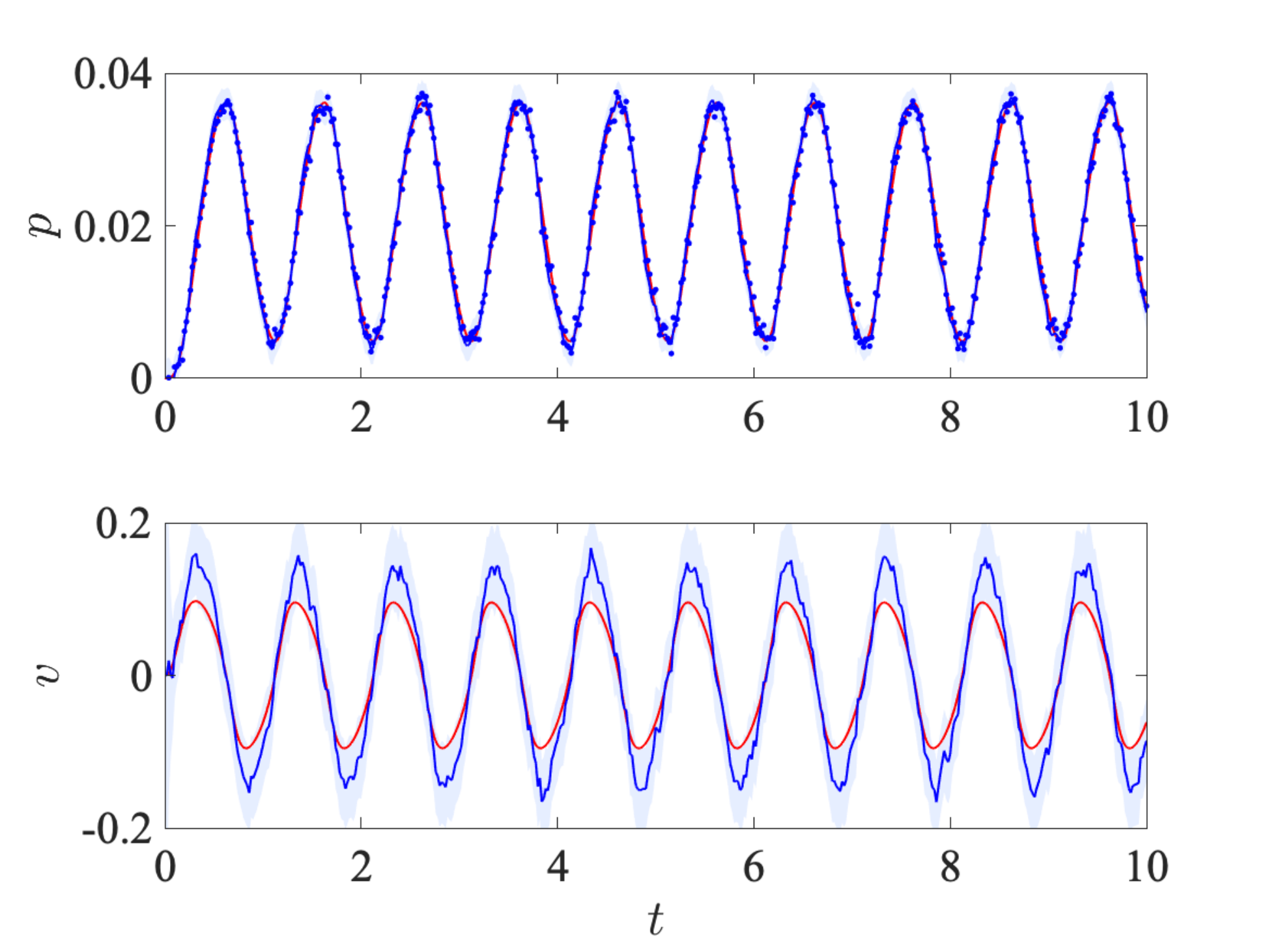}
        }
    }
    \centerline{
        \subfigure[Adaptive learning Kalman filter]{
            \includegraphics[width=0.9\columnwidth]{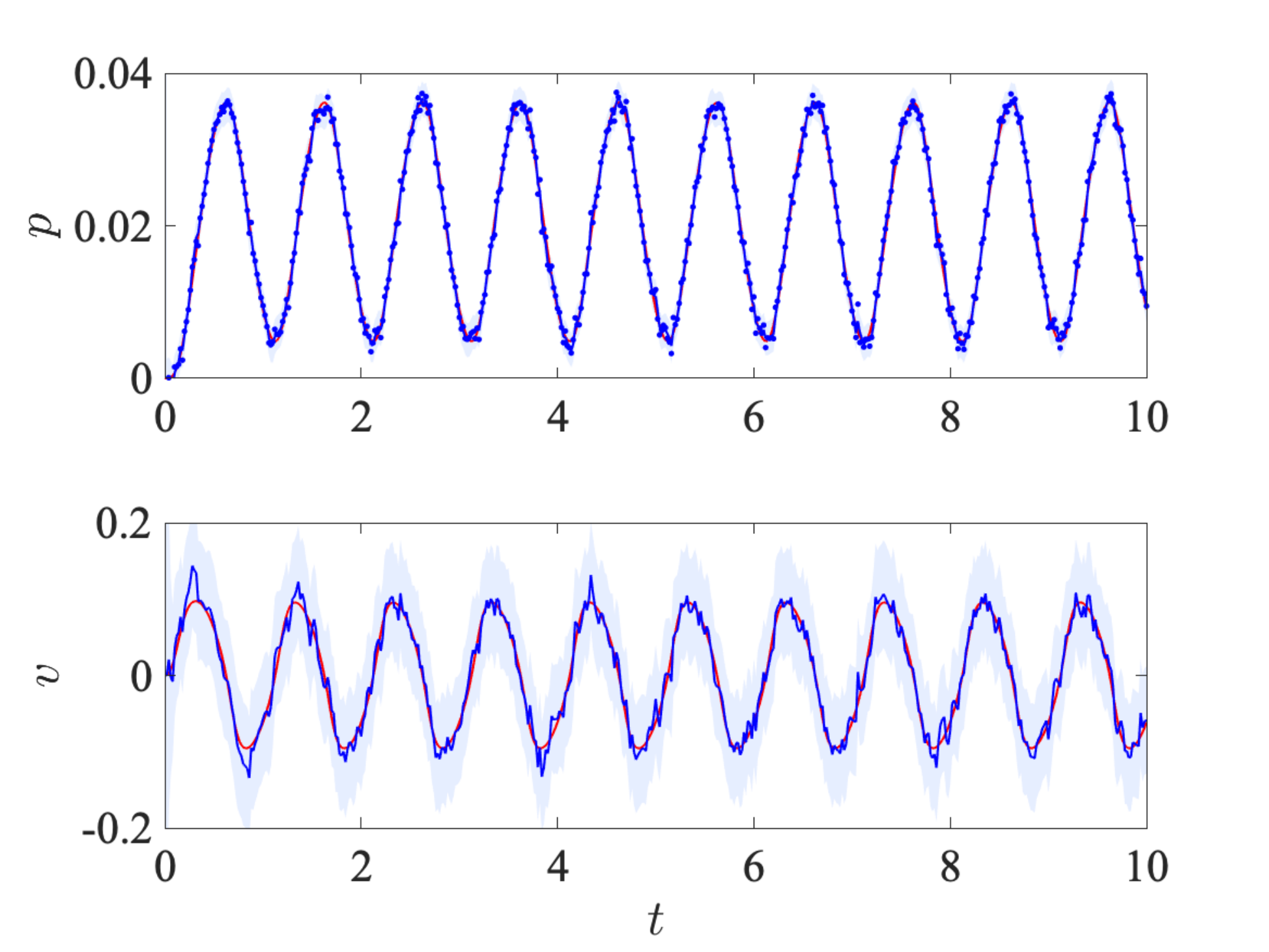}
        }
    }
    \caption{Simulation results: the estimated position and velocity are illustrated by blue curves with $3\sigma$ bounds, against the true trajectory illustrated by red curves. The dots at the position figure correspond to the position measurements. Adaptive learning Kalman filter results in smaller estimation errors, especially for $v$}\label{fig:NE}
\end{figure}

\begin{table}
    \caption{Mean squared error comparison}\label{tab:err}
    \vspace*{-0.2cm}
    \begin{center}
        \begin{tabular}{ccc}\toprule
        & Position est. error &  Velocity est. error \\\midrule
            KF & \num{4.6832e-5} & \num{1.4911e-3} \\
            ALKF & \num{2.7267e-5} & \num{5.6530e-4} \\ \bottomrule
        \end{tabular}
    \end{center}
\end{table}

The corresponding simulation results are illustrated at Figure \ref{fig:NE}, where the performance of the proposed adaptive learning Kalman filter is compared with the Kalman filter. 
There is no clear difference in the position estimation, as it is measured directly by a relatively accurate sensor. 
The advantage of the adaptive learning Kalman filter become more noticeable for the velocity estimate: the velocity estimated by the Kalman filter overshoots repeatedly; such behaviors dissipate gradually for the adaptive learning Kalman filter, and it follows the true velocity relatively well as the time progresses. 
The difference are clearly depicted by the mean squared errors summarized at Table \ref{tab:err}.

Finally, the progressive learning of the extended Gaussian process  is illustrated at Figure \ref{fig:EGP_Delta} for varying time instances.
These show that the accuracy and the confidence level of the extended Gaussian process increase over time as more training data become available. 
The learning model can be utilized beyond the presented estimation scenario.
For example, it would improve the accuracy of the estimate for any other trajectories in the similar operating range, and it can be utilized for feedback controls as well.

\begin{figure}
    \centerline{
        \subfigure[$t=0$]{
            \includegraphics[width=0.50\columnwidth]{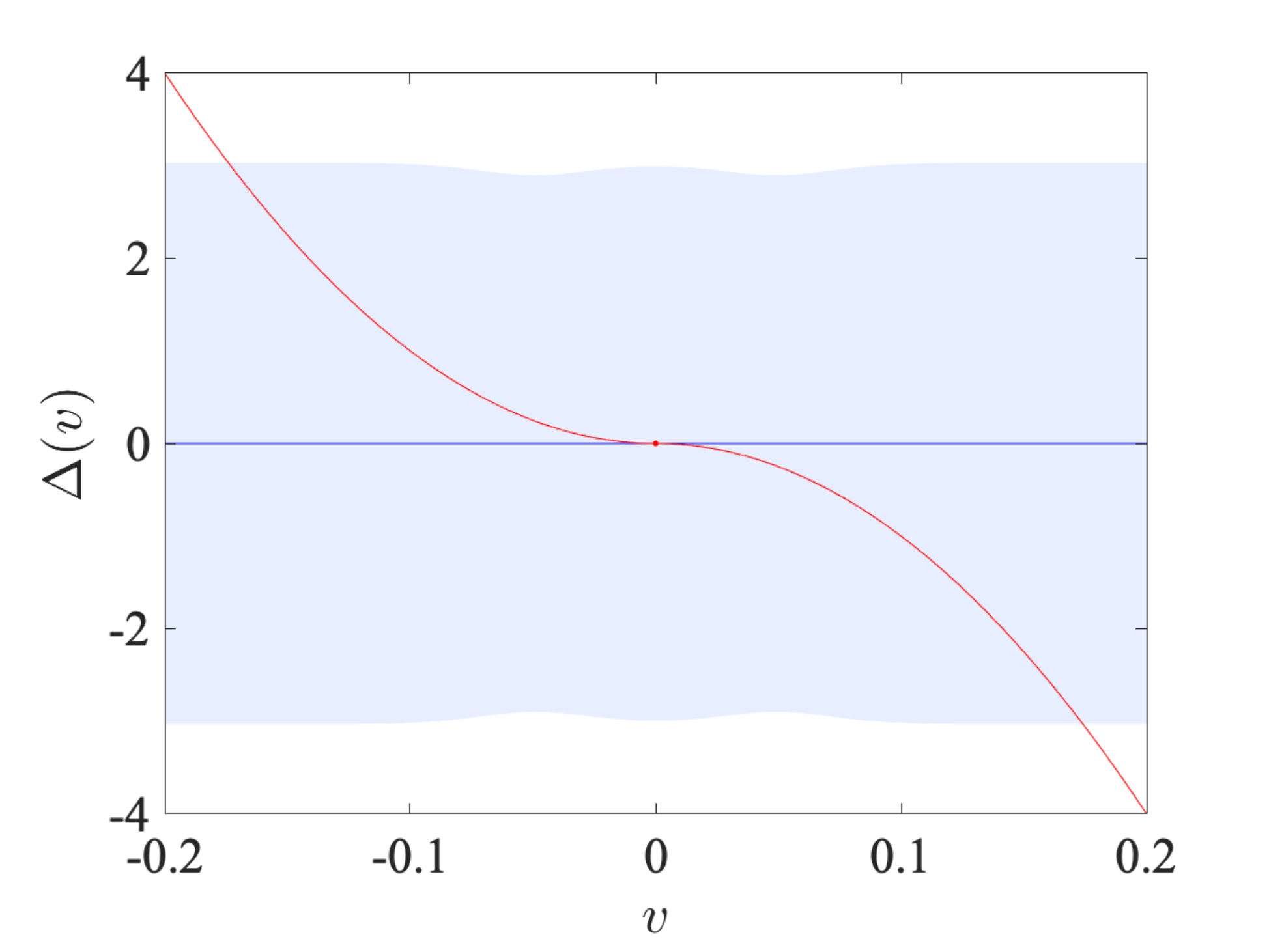}
        }
        \subfigure[$t=0.4$]{
            \includegraphics[width=0.50\columnwidth]{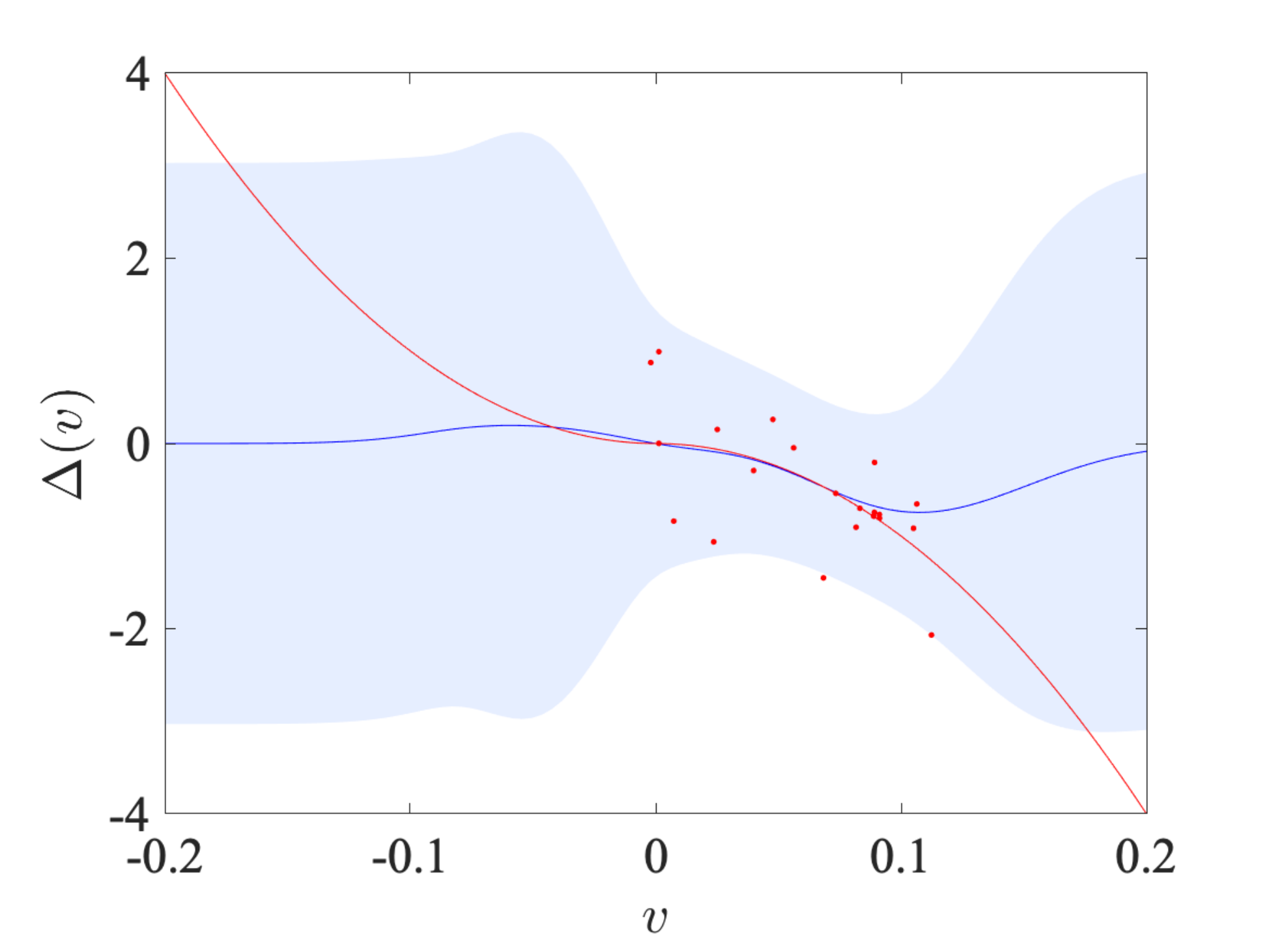}
        }
    }
    \centerline{
        \subfigure[$t=0.8$]{
            \includegraphics[width=0.50\columnwidth]{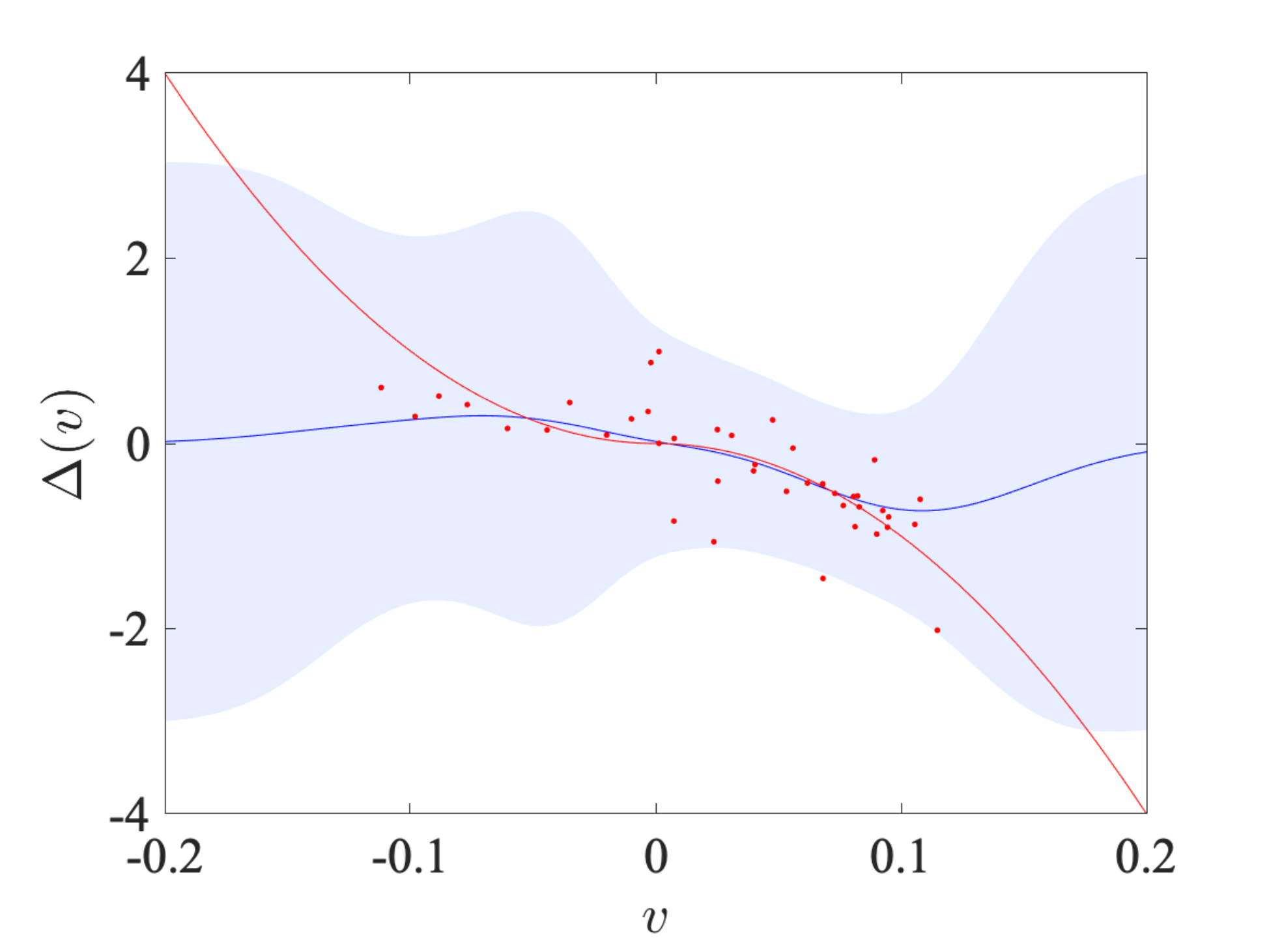}
        }
        \subfigure[$t=1.2$]{
            \includegraphics[width=0.50\columnwidth]{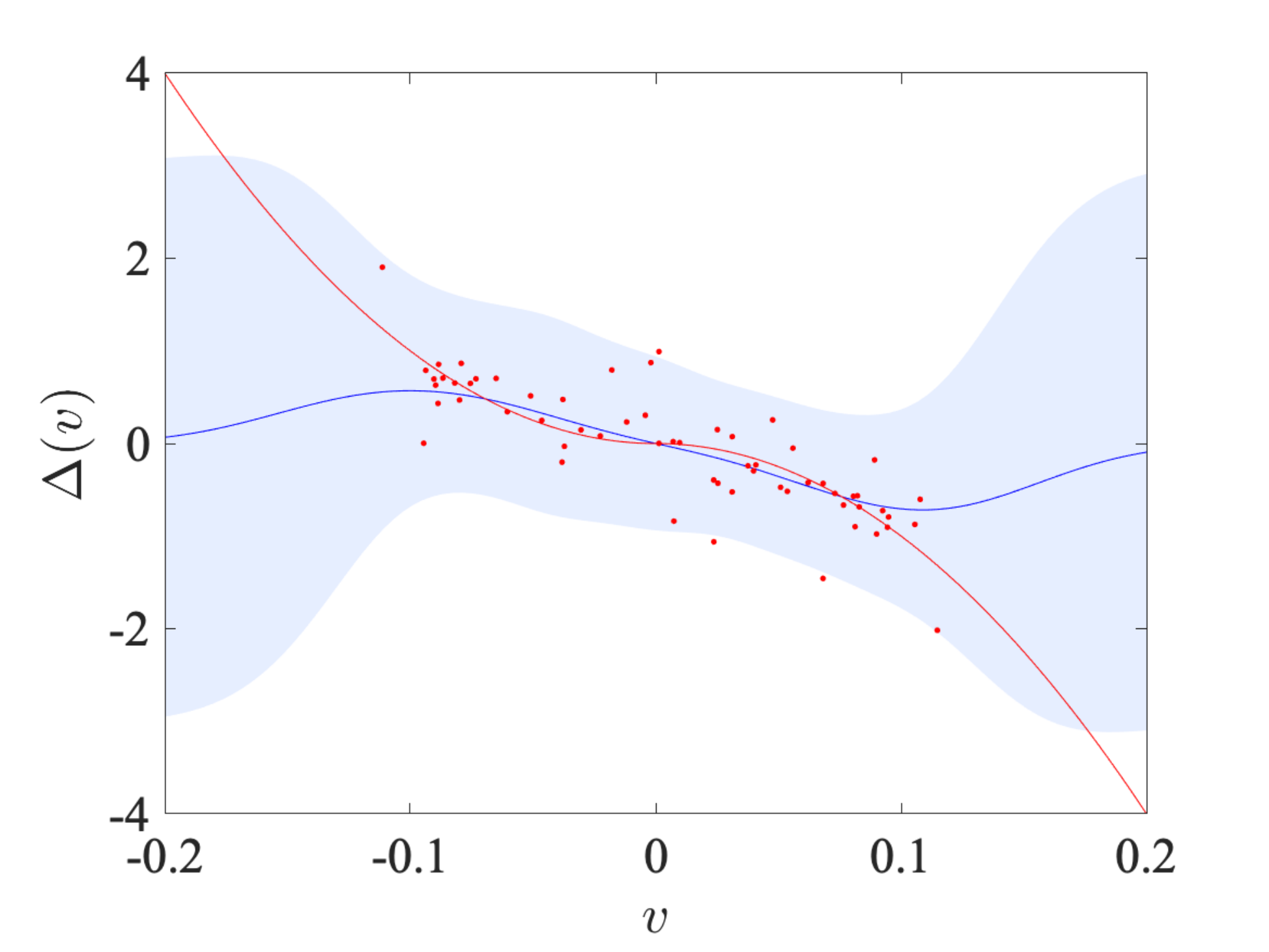}
        }
    }
    \centerline{
        \subfigure[$t=1.6$]{
            \includegraphics[width=0.50\columnwidth]{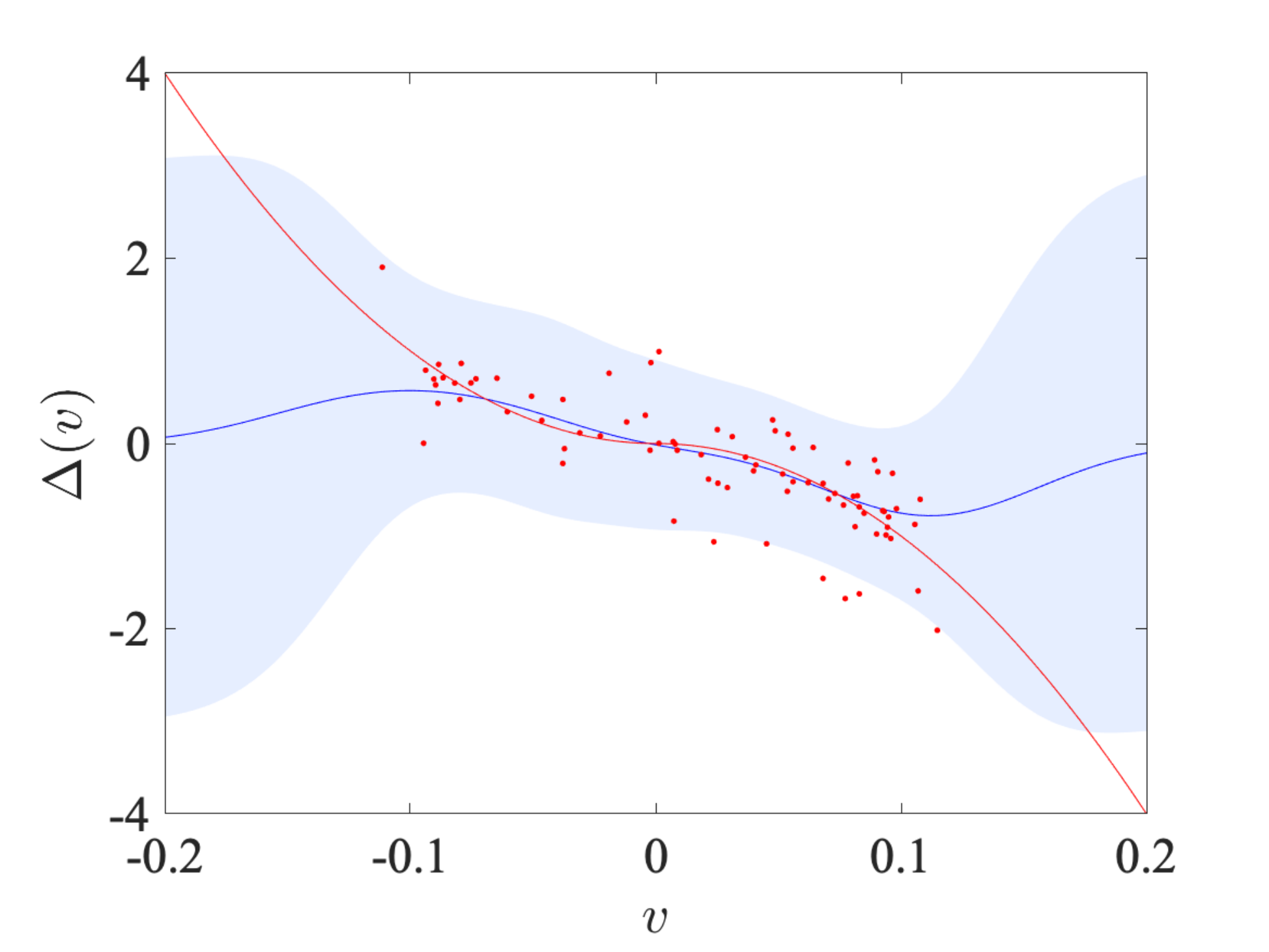}
        }
        \subfigure[$t=2.0$]{
            \includegraphics[width=0.50\columnwidth]{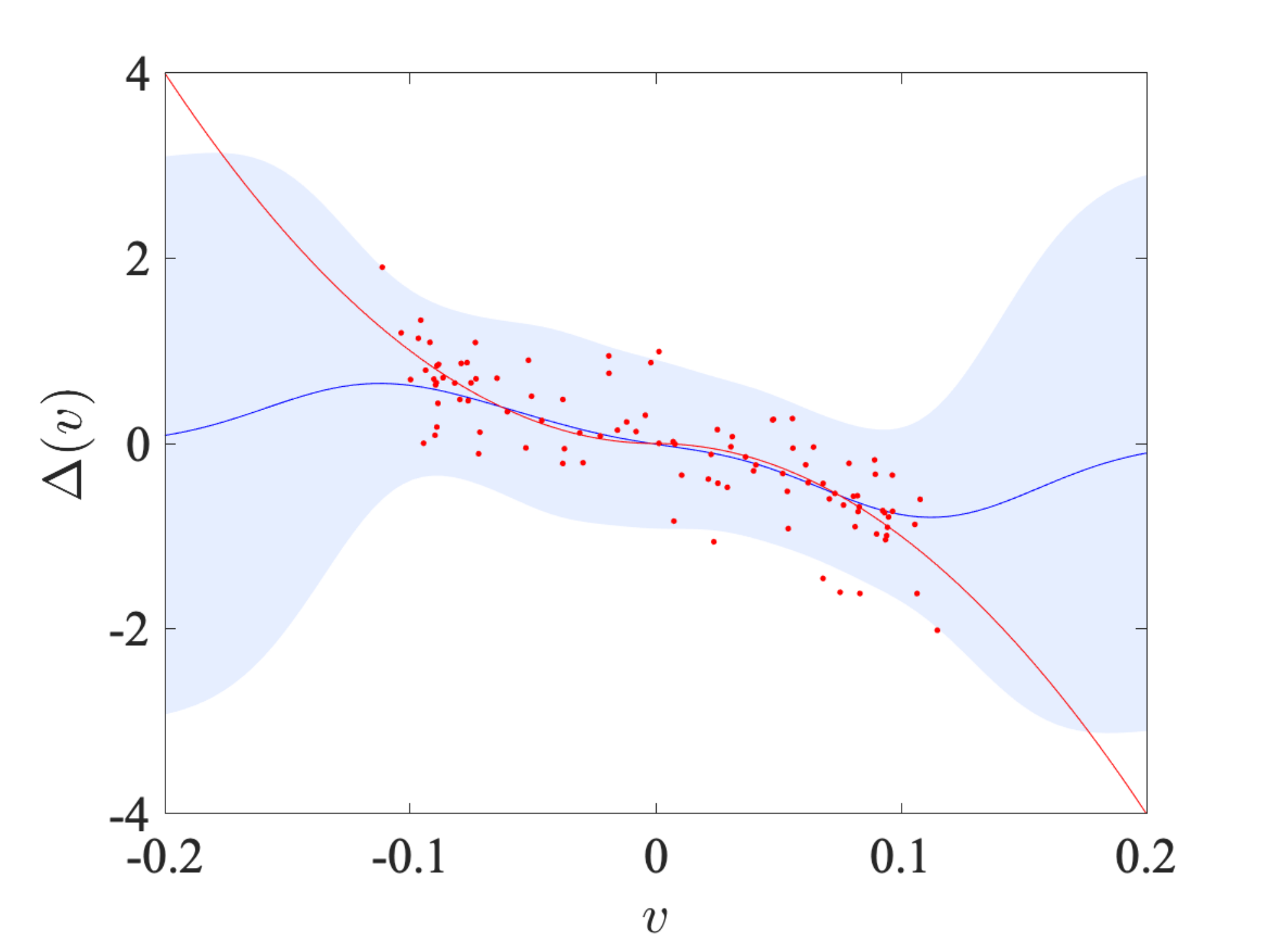}
        }
    }
    \caption{Progressive learning of extended Gaussian process (GP): the output of GP is illustrated by blue curves with $3\sigma$ bounds for varying $v$ with $x=0$, against the true $\Delta(v)$ illustrated by red curves. The red dots represent the training data. 
    As the time progresses, more data become available. Consequently, the learned model becomes gradually closer to the true value with an increased confidence level, which contributes to the improved accuracy of the adaptive learning Kalman filter.}\label{fig:EGP_Delta}
\end{figure}

\section{Conclusions}

We have presented an adaptive learning Kalman filter where the unknown disturbance is modeled as a Gaussian process.
This exhibits a unique feature of accounting uncertainties in the concurrent estimate of the state and the disturbance.
The future works include optimization of hyperparameters and sparsification for the  Gaussian process.

\appendix

\subsection{Properties of Gaussian Distribution}
Let $x\in\Re^n,y\in\Re^n$ be jointly Gaussian with
\begin{align*}
    \begin{bmatrix} 
        x \\ y
    \end{bmatrix}
    \sim
    \mathcal{N} \left(
        \begin{bmatrix}
            a \\ b
        \end{bmatrix}
        ,
        \begin{bmatrix}
            A & C \\
            C^T & B
        \end{bmatrix}
    \right),
\end{align*}
where $a\in\Re^n$, $b\in\Re^m$ are the mean values, and the matrices $A\in\Re^{n\times n}$, $B\in\Re^{m\times m}$, and $C\in\Re^{n\times m}$ are covariance matrices. 
The marginal distribution for $x$ is simply $x\sim\mathcal{N}(a, A)$, and the conditional distribution $x|y$ is 
\begin{align}
    x|y & \sim \mathcal{N} (a + CB^{-1}(y-b), A- C B^{-1} C^T).\label{eqn:Gau_cond}
\end{align}

Next, let $x\sim\mathcal{N}(a, A)$ and $y|x\sim\mathcal{N}(Hx+ c, B)$ for $H\in\Re^{m\times n}$ and $c\in\Re^m$.
The joint distribution is
\begin{align}
    \begin{bmatrix} 
        x \\ y
    \end{bmatrix}
    \sim
    \mathcal{N} \left(
        \begin{bmatrix}
            a \\ Ha + c
        \end{bmatrix}
        ,
        \begin{bmatrix}
            A & AH^T \\
            HA & HAH^T + B
        \end{bmatrix}
    \right).\label{eqn:Gau_joint}
\end{align}

\subsection{Squared exponential kernel}

The squared exponential kernel is defined as
\begin{align*}
    \kerG(x_i, x_j) = \sigma_f^2 \exp( - \frac{1}{2}(x_i-x_j)^T L^{-1} (x_i-x_j)) + \delta_{i,j} \sigma_n^2,
\end{align*}
for $\sigma_f, \sigma_n>0$ and a positive-definite symmetric matrix $L=L^T\in\Re^n$ that determines the characteristic length scale.



\end{document}